\documentclass{article}
\usepackage{graphicx} 

\usepackage{amsmath,amssymb,amscd,amsthm,latexsym, graphics}
\usepackage[applemac]{inputenc}
\usepackage[all]{xy}
\usepackage{pstricks-add}
\newtheorem{thm}{Theorem}
\newtheorem{lem}[thm]{Lemma}
\newtheorem{prop}[thm]{Proposition}
\newtheorem{cor}[thm]{Corollary}
\newtheorem{defin}[thm]{Definition}
\newtheorem{rem}[thm]{Remark}

\newtheorem{ex}[thm]{Example}

\def\nil{{\rm{nil}\hskip1pt}}
\def\cat{{\rm{cat}\hskip1pt}}
\def\wcat{{\rm{wcat}\hskip1pt}}
\def\secat{{\rm{secat}\hskip1pt}}
\def\wsecat{{\rm{wsecat}\hskip1pt}}
\def\TC{{\rm{TC}\hskip1pt}}
\def\wTC{{\rm{wTC}\hskip1pt}}
\def\cuplength{{\rm{cuplength}\hskip1pt}}

\newdir{ >}{{}*!/-7pt/\dir{>}}

\def\Z{{\mathbb{Z}}}

\def\cat{{\rm{cat}\hskip1pt}}
\def\wcat{{\rm{wcat}\hskip1pt}}

\def\secat{{\rm{secat}\hskip1pt}}

\def\nil{{\rm{nil}\hskip1pt}}
\def\TC{{\rm{TC}\hskip1pt}}
\def\wTC{{\rm{wTC}\hskip1pt}}

\def\zcl{{\rm{zcl}\hskip1pt}}

\begin{document}

\title{A study of the Ganea conjecture for topological complexity by using weak topological complexity}

\author{ J.M. Garcia
Calcines\footnote{Universidad de La Laguna, Facultad de
Matem\'aticas, Departamento de Matem\'aticas, Estadistica e I.O. and Instituto de Matem\'aticas y Aplicaciones (IMAULL), 38271 La Laguna, Spain. E-mail:
 \texttt{jmgarcal@ull.es}}
\,and L. Vandembroucq\footnote{Centro
 de Matem\'{a}tica, Universidade do Minho, Campus de Gualtar,
 4710-057 Braga, Portugal. E-mail: \texttt{lucile@math.uminho.pt}}
}

\date{\today}

\maketitle

\begin{abstract}
{In this paper, we provide sufficient conditions for a space $X$ to satisfy the Ganea conjecture for topological complexity. To achieve this, we employ two auxiliary invariants: weak topological complexity in the sense of Berstein-Hilton, along with a certain stable version of it. Several examples are discussed.}
\end{abstract}

\section*{Introduction}

The topological complexity of a space $X$, denoted $\mbox{TC}(X),$ is defined as the sectional category (or Schwarz genus) of the evaluation fibration $\pi _X:X^I\rightarrow X\times X,$
$$\pi (\alpha ):=(\alpha (0),\alpha(1))$$
This numerical homotopy invariant was defined by M. Farber \cite{Far} in order to study the measure of discontinuity of any motion planning algorithm in Robotics. Topological complexity is related to Lusternik-Schnirelmann category, as the latter may be also defined as the sectional category of the path fibration $PX\rightarrow X$.

In this paper, our focus is on exploring the Ganea conjecture in the context of topological complexity, which is a variation of the classical Ganea conjecture for Lusternik-Schnirelmann category \cite{Ganea1}. The original Ganea conjecture posed the question of whether the equality $\mbox{cat}(X\times S^k)=\mbox{cat}(X)+1$ holds true for any finite CW-complex and $k\geq 1$. The affirmation of this equality is known as the classical Ganea conjecture, which remained open until Iwase \cite{Iwase} provided a set of counterexamples refuting it. Nonetheless, over the years of studying this conjecture, several positive outcomes have emerged. In particular, the equality $\mbox{cat}(X\times S^k)=\mbox{cat}(X)+1$ has been established for simply connected rational spaces (\cite{Hess}, \cite{Jessup}) and for certain classes of manifolds and CW-complexes (\cite{Rudyak}, \cite{Singhof}, \cite{strom-ls-on-manifolds}). In the sequel of Iwase's result, further works have been realized (for instance, \cite{Ainfinity}, \cite{Stanley}, \cite{Stanley2}, \cite{Qcat}) in order to exhibit additional classes of counter-examples and of spaces satisfying the Ganea conjecture.

The $\TC$-Ganea conjecture, or Ganea conjecture for topological complexity, seeks to determine whether the following equality holds:
$$\mbox{TC}(X\times S^k)=\mbox{TC}(X)+\mbox{TC}(S^k)$$
This conjecture explores the relationship between the topological complexity of a finite CW-complex $X$ and the topological complexity of the product space $X\times S^k$, where $k\geq 1$. It is worth noting that $\TC(S^k)$ equals 1 when $k$ is odd and 2 when $k$ is even.
Significant progress has been made in establishing positive results for this conjecture in the realm of rational homotopy theory. Jessup, Murillo, and Parent \cite{JMP} as well as Carrasquel \cite{carrasquel} have made noteworthy advancements in this direction.
Conversely, in their work \cite{GonzalezGrantV}, J. Gonzalez, M. Grant, and L. Vandembroucq provided a counterexample for the case where $k$ is even. However, the situation for odd values of $k$ remains unresolved and awaits further investigation.

In this paper, we aim to establish sufficient conditions for a space $X$ to satisfy the $\TC$-Ganea conjecture. To accomplish this, we initially focus on the case where the dimension is odd, that is, $k=2n+1$, and build on the approach developed by Strom in \cite{strom-ls-on-manifolds} where the Berstein-Hilton weak-category was used as an auxiliary invariant. Here our auxiliary invariant will be the "weak topological complexity" ($\mbox{wTC}(-)$), which was introduced in \cite{Weaksecat}.
Our first main result, as stated in Theorem \ref{result}, demonstrates that under certain conditions, we have $\wTC(X\times S^{2n+1})\geq \wTC(X)+1$. To prove this result, we utilize the crucial relationship between weak topological complexity and the weak category of the homotopy cofiber of the diagonal map $\Delta :X\rightarrow X\times X$, which was established in  \cite{Weaksecat}.
By establishing the inequality $\wTC(X\times S^{2n+1})\geq \wTC(X)+1$, we can deduce that $\TC(X\times S^{2n+1})=\TC(X)+1$, provided that the approximation $\wTC(X)$ is optimal, meaning that $\wTC(X)=\TC(X)$.
By employing a similar strategy for the even dimensional case, we are able to derive our second key result, as presented in Theorem \ref{result2}. However, in the case of even dimensions, we need to introduce a stable version of the auxiliary invariants, namely $\wcat$ and $\wTC$, beforehand.
Subsequently, we delve into the study of a specific class of two-cell complexes and utilize Hopf invariants to establish conditions that guarantee the equality between our auxiliary invariants and $\TC$. This exploration leads us to identify classes of spaces that satisfy the $\TC$-Ganea conjecture. It is worth noting that Hopf invariants have played a pivotal role in Iwase's disproof of the classical Ganea conjecture \cite{Iwase}. In the context of topological complexity, techniques involving Hopf invariants have been developed in \cite{GonzalezGrantV} and \cite{GGV2}, with particular applications to the class of two-cell complexes $X=S^p\cup_{\alpha}e^{q+1}$ in the metastable range ($2p-1 < q < 3p-2$). It is within this class of spaces that we focus our study, making use of results from the aforementioned articles \cite{GonzalezGrantV} and \cite{GGV2}.

The paper is organized as follows. In the first section, we provide essential notation, definitions, and results that will be crucial for our main findings. This includes a review of the concepts of Berstein-Hilton weak category and weak topological complexity, along with their key properties. Additionally, we present relevant information on Berstein-Hilton Hopf invariants and the homotopy cofiber of the diagonal map for spheres. The second section of the paper focuses on presenting the proofs of our main results. Here, we establish sufficient conditions under which the $\TC$-Ganea conjecture is satisfied. We provide detailed explanations and justifications for these conditions. Finally, the paper concludes with an application of our results to the study of two-cell complexes in the metastable range. We demonstrate how our findings can be applied to this specific class of spaces and discuss the implications of our results in this context.

\section{Preliminaries}\label{preliminar}

Throughout the paper, we shall work in the category of well-pointed
compactly generated Hausdorff spaces. We assume that the
reader is familiarized with the notions of sectional category
and its main specializations: topological
complexity and Lusternik-Schnirelmann category. For fundamental aspects on these invariants see \cite{C-L-O-T}, \cite{Sch}, \cite{Far} .

\subsection{Weak category and weak topological complexity}

We here recall some material from \cite{Weaksecat}, referring the reader to this paper for more details.

Recall that, given any map $f:E\rightarrow B$, there is a
Whitehead-type characterization of sectional category, which is
given as follows. Consider the \emph{$n$-sectional fatwedge} of
any map $f:E\rightarrow B,$ $\kappa _n:T^n(f)\rightarrow B^{n+1},$
inductively defined by starting with $\kappa _0=f:E\rightarrow B$
and considering the (fibrewise) join of $\kappa _{n-1}\times id_B$ and
$id_{B^n}\times f.$
Then $\secat(f)\leq n$ if and only
if there is, up to homotopy, a lift of the $(n+1)$-diagonal map $\Delta_{n+1}=\Delta^B_{n+1}:B \to B^{n+1}$.
$$\xymatrix{
{} & {T^n(f)} \ar[d]^{\kappa _n} \\
{B} \ar@{.>}[ur] \ar[r]_{\Delta _{n+1}} & {B^{n+1}} }$$

For any $n\geq 0$, we shall denote by $B^{\wedge n}$ the $n$-fold
smash-product, by $q_{n}=q^B_{n}:B^{n}\to B^{\wedge n}$ the
identification map and by $\bar \Delta_{n}=\bar \Delta^B_{n}:B
\stackrel{\Delta_{n}}{\to} B^{n}\stackrel{q_{n}}{\to} B^{\wedge
n}$ the reduced $n$-diagonal. Recall that the weak category of $B$,
$\wcat(B)$, introduced by Berstein and Hilton \cite{B-H}, is the
least integer $n$ such that $\bar \Delta_{n+1}$ is homotopically
trivial. This is a lower bound for the L-S category of $B$ since
$\cat(B)\leq n$ if and only if the diagonal $\Delta^B_{n+1}:B\to
B^{n+1}$ lifts, up to homotopy, in the fatwedge
$T^n(B)=\{(b_0,b_1,...,b_n)\in B^{n+1}:b_i=*,\hspace{3pt}
\mbox{for some}\hspace{3pt}i\}$ and $B^{\wedge
(n+1)}=B^{n+1}/T^n(B).$

Let $f: E\to B$ be a map and, for any integer $n$, let
$C_{\kappa_n}$ be the homotopy cofibre of the $n$-sectional fat
wedge of $f$, $\kappa _n:T^n(f)\rightarrow B^{n+1}$. If
$l_n:B^{n+1}\rightarrow C_{\kappa_n}$ denotes the induced map,
then the weak sectional category of $f$, denoted by
$\mbox{wsecat}(f)$, is the least integer $n$ (or $\infty$) such
that the composition
$$\xymatrix{ B\ar[r]^{\Delta_{n+1}}& B^{n+1}\ar[r]^{l_n}&
C_{\kappa_n}}$$ is homotopically trivial. This is a lower bound of
sectional category whose most important properties are summarized
in the following result.

\begin{prop}\label{properties}\cite[Th. 21]{Weaksecat}
Let $f: E\to B$ be a map and $C_f$ be its homotopy cofibre. Then
\begin{enumerate}
\item[(a)] $\max\{\wcat(C_f)-1,\nil \ker f^*\} \leq \wsecat(f)\leq \min\{\wcat(B),\wcat(C_f)\}$
\item[(b)] If the map $f: E\to B$ admits a homotopy retraction (i.e.,
there is a map $r:B\rightarrow E$ such that $rf\simeq 1_E$), then
$$\wsecat(f)=\wcat(C_f)\hspace{5pt}\mbox{and}\hspace{5pt}\nil \ker f^*=\cuplength(C_f).$$
\end{enumerate}
\end{prop}

Here $\nil \ker f^*$ denotes the nilpotency of the kernel of the
morphism $f^*$ which is induced by $f$ in cohomology (where the coefficients can be taken in any commutative ring). This is a
classical lower bound for sectional category \cite{Sch}. As a consequence of Proposition \ref{properties}(a), we note that $\wsecat(f)=\mbox{secat}(f)$ when
$\mbox{secat}(f)=\nil \ker f^*.$ Another conditions which ensure the equality $\wsecat(f)=\mbox{secat}(f)$ are given in the following proposition established in \cite{Weaksecat}:

\begin{prop}\label{secat-wsecat}
Let $f:E\rightarrow B$ be any map, where $E$ and $B$ are
$(p-1)$-connected CW-complexes ($p\geq 1$). If $\dim(B)=N$ and
either one of the following conditions is satisfied
\begin{enumerate}
\item[(i)] $N\leq p(\wsecat(f)+2)-2$

\item[(ii)] $N\leq p(\secat(f)+1)-2,$
\end{enumerate}

\noindent then $\secat(f)=\wsecat(f).$
\end{prop}

Note that in \cite{Weaksecat} this result was stated for $p\geq 2$. Since we always have $\secat(f)\leq \cat(B)\leq N$, the statement is actually also true for $p=1$.

\bigskip
In this paper, we are more interested in the particular case of \emph{weak topological
complexity}, denoted by $\mbox{wTC}(X).$ This is a lower bound of topological
complexity defined as $\mbox{wTC}(X):=\mbox{wsecat}(\pi )$ where $\pi :X^I\rightarrow
X\times X$ denotes the evaluation path fibration. As $\pi $ is the fibration associated with the diagonal map $\Delta :X\rightarrow
X\times X,$ we also have that $\mbox{wTC}(X)=\mbox{wsecat}(\Delta
).$ As a consequence of Theorem \ref{properties} we have the following equality:
$$\mbox{wTC}(X)=\mbox{wcat}(C_{\Delta }(X))$$ \noindent where
$C_{\Delta }(X)$ stands for the homotopy cofibre of the diagonal map
$\Delta :X\rightarrow X\times X$. We note that the equality $\mbox{TC}(X)=\mbox{cat}(C_{\Delta }(X))$ holds for several classes of spaces (see \cite{GC-V}, \cite{GGV2}). However, it is not true in general as established by Dranishnikov (\cite{Dr}).


As a direct consequence of Proposition \ref{secat-wsecat}, we
have the following corollary:

\begin{cor}\label{resul3}{\rm
Let $X$ be any $(p-1)$-connected CW-complex $(p\geq 1$). If
$\mbox{dim}(X)=N$ and either one of the following conditions is
satisfied:
\begin{enumerate}
\item[(i)] $N\leq (p(\mbox{wTC}(X)+2))/2-1;$ or

\item[(ii)] $N\leq (p(\mbox{TC}(X)+1))/2-1.$
\end{enumerate}
\noindent Then $\mbox{TC}(X)=\mbox{wTC}(X).$}
\end{cor}

\begin{rem}{\rm
As pointed out for the general case of sectional category, we also
have a cohomological condition ensuring the equality $\mbox{TC}(X)=\mbox{wTC}(X).$
Indeed, if $\tilde H^*$ stands for the reduced cohomology with
coefficients in a field $\Bbbk$, then
$\nil \ker \Delta ^*$ is exactly $\zcl_{\Bbbk}(X)$, the \textit{zero divisor cup length}  introduced by Farber in \cite{Far}. Then, from Proposition
\ref{properties}(b) we have that $\mbox{TC}(X)=\mbox{wTC}(X)$ as
far as $\mbox{TC}(X)=\zcl_{\Bbbk}(X) .$ In \cite{Weaksecat} and \cite{GC-V} , we have shown that there exist spaces $X$ for which $\mbox{TC}(X)=\wTC(X)>\zcl_{\Bbbk}(X) .$ More precisely, the space $X=S^3\cup _{\alpha }e^7$, where $\alpha\in \pi_6(S^3)$ is the Blakers-Massey element (that is, $X$ is the 7-skeleton of
$\mbox{Sp}(2)$), satisfies
$${\rm zcl}_{\Bbbk}(X)=2 \quad \wTC(X)=\TC(X)=3.$$
}

\end{rem}

In the sequel, we will also use the following result (see \cite{Weaksecat}).

\begin{prop}\label{growth-of-wcat}
	Let $f:A\to B$ be a map between $(p-1)$-connected CW-complexes
	($p\geq 1$). If $f$ is an $r$-equivalence, $\wcat(A)\geq k$ and
	$\dim(A)\leq r+p(k-1)-1$ then $\wcat(B)\geq \wcat(A)$.
\end{prop}

%
%

\subsection{Hopf invariants}

In \cite{B-H}, Berstein and Hilton have introduced two notions of (generalized) Hopf invariants to study the increment of the category/weak-category upon cell attachments. We will refer to these invariants as the $\cat$-Hopf invariants and \textit{crude} Hopf invariants, respectively. For the sake of simplicity, we will consider the following conditions, which are sufficient for our purpose. Let $X=K\cup_{\alpha }CS$ where $K$ is a $(p-1)$-connected CW complex with $p\geq 2$, $\cat(K)\leq k$ and $\dim (K) \leq (k+1)p-2$. We also assume that $S$ is a finite wedge of spheres of the same dimension $d$ satisfying $d\geq \dim(K)$.

Given an attaching map $\alpha$, the definition of the Berstein-Hilton Hopf invariants depends on the choice of a homotopy lifting $\phi:K\to T^k(K)$ of the diagonal map $\Delta^K_{k+1}:K\to K^{k+1}$. However, under the condition $\dim (K) \leq (k+1)p-2$, there is a unique such lifting and the Hopf invariants depend only on the homotopy class of $\alpha$. We denote by $H(\alpha)\in [(CS,S),(K^{k+1},T^k(K))]\cong [S,F_k(K)]$ the $\cat$-Hopf invariant of $\alpha$ (which can be seen as the obstruction to extend to $X$ the lifting $\phi$ of $\Delta^K_{k+1}$ -- see \cite{B-H} or \cite{C-L-O-T} for the precise definition). Here $F_k(K)$ denotes the homotopy fibre of the inclusion $T^k(K)\hookrightarrow K^{k+1}$. The crude Hopf invariant of $\alpha$ is given by $\bar H(\alpha):=(q_{k+1})_*H(\alpha)\in [\Sigma S,K^{\wedge(k+1)}]$ where $(q_{k+1})_*$ is the morphism induced by the identification map $(K^{k+1},T^k(K))\to (K^{\wedge(k+1)},\ast)$.

Under the conditions above, it follows from \cite{B-H} that
\begin{enumerate}
	\item $\cat(X) \leq k $ if, and only if, $H(\alpha)=0$,
	\item $\wcat(X) \leq k $ if, and only if, $\bar H(\alpha)=0$.
\end{enumerate}
As it will be useful in this work, we detail the second statement through the following lemma. We include a proof which will lead to a generalization in a subsequent section (Lemma \ref{crudeHsquare2}).
\begin{lem}\label{crudeHsquare} Let $K$ be a $(p-1)$-connected CW complex with $p\geq 2$, $\cat(K)\leq k$ and $\dim (K) \leq (k+1)p-2$. Let $d\geq \dim (K)$ and $\alpha: S\to K$ be a map where $S$ is a finite wedge of spheres $S^d$. Then, for $X=K\cup_{\alpha }CS$, there exists a homotopy commutative diagram
\[\xymatrix{X \ar[rr]^{\bar{\Delta}_{k+1}}\ar[d]_{\delta} && X^{\wedge (k+1)}\\
\Sigma S \ar[rr]_{\bar H(\alpha)}&& K^{\wedge (k+1)}\ar@{^(->}[u]}\]
where $\delta$ is the connecting map in the Puppe sequence of the cofibration $S\stackrel{\alpha}{\to}K\to X$. Moreover $\wcat(X)\leq k$ if, and only if, $\bar H(\alpha)=0$.
\end{lem}

\begin{proof}
The diagram can be deduced from \cite[Proposition 3.7(ii)]{B-H}, see also \cite[Lemma 6.33]{C-L-O-T}. It is clear that $\bar H(\alpha)=0$ implies that $\wcat(X)\leq k$. Conversely, let us suppose that $\wcat(X)\leq k$. Then $\bar{\Delta}_{k+1}\simeq \ast$. Note that the inclusion $K^{\wedge (k+1)}\hookrightarrow X^{\wedge (k+1)}$ is a $(kp+d)$-equivalence. Since $\dim(X)=d+1< kp+d$, we deduce from $\bar{\Delta}_{k+1}\simeq \ast$ that $\bar H(\alpha)\delta\simeq \ast$. Considering the Puppe sequence of the cofibration $S\stackrel{\alpha}{\to}K\to X$ we then get a map $\Sigma K \to  K^{\wedge (k+1)}$ making homotopy commutative the following diagram
	\[\xymatrix{X\ar[d]_{\delta} \\
		\Sigma S \ar[rr]^{\bar H(\alpha)}\ar[d]_{\Sigma \alpha}&& K^{\wedge (k+1)}\\
		\Sigma K\ar[rru]
	}\]
	Since $K^{\wedge (k+1)}$ is $(k+1)p-1$-connected and $\dim (K) \leq (k+1)p-2$ we see that this is map is homotopically trivial and consequently $\bar H(\alpha)=0$.	
\end{proof}

We will also use the following observations: 

\begin{rem}\label{rem-inv}{\rm
Let us consider the commutative diagram corresponding to the identification map $(K^{k+1},T^k(K))\to (K^{\wedge(k+1)},\ast)$. If we denote by $\bar q_{k+1}: F_k(K) \to \Omega K^{\wedge(k+1)}$ the induced map at the level of the homotopy fibers, then, through adjunction, $\bar H(\alpha)$ can be identified to the following composite:
\[
\Sigma S \stackrel{\Sigma H(\alpha)}{\longrightarrow} \Sigma F_k(K) \stackrel{\Sigma \bar q_{k+1}}{\longrightarrow } \Sigma \Omega K^{\wedge(k+1)} \stackrel{ev}{\longrightarrow } K^{\wedge(k+1)}
\]
where $ev$ is the evaluation map. When $K=S^p$ or more generally any ($p-1$)-connected CW-complex with $p\geq 2$, we can use a Blakers-Massey argument to see that the map
\[
\Sigma F_k(K) \stackrel{\Sigma \bar q_{k+1}}{\longrightarrow} \Sigma \Omega K^{\wedge(k+1)} \stackrel{ev}{\longrightarrow} K^{\wedge(k+1)}
\]
is a $p(k+2)-1$ equivalence.
}
\end{rem}

\begin{rem}\label{Hil-inv}{\rm
In the specific case of a map $\alpha: S \to S^p$ ($p\geq 2$), where $S$ is a finite wedge of spheres $S^d$ with $d < 3p-2$, the cat-Hopf invariant $H(\alpha)$ exhibits stability and is entirely determined by its projection $H_0(\alpha): S \to S^{2p-1}$ onto the bottom sphere of $$F_1(S^p)=\Omega S^p \ast \Omega S^p \simeq S^{2p-1} \vee S^{3p-2} \vee S^{3p-2} \vee \cdots$$
In this case we also have $\bar H(\alpha)=\Sigma H_0(\alpha)$.}
\end{rem}

\subsection{Cofibre of the diagonal map of spheres}

Recall that we denote the homotopy cofiber of the diagonal map $\Delta: X \rightarrow X \times X$ for any space $X$ as $C_{\Delta}(X)$. In this subsection, we will recall important facts on $C_{\Delta}(S^n)$ and its reduced diagonal $\bar{\Delta}: C_{\Delta}(S^n) \rightarrow C_{\Delta}(S^n) \ C_{\Delta}(S^n)$. 

As shown in \cite[Prop. 28]{Weaksecat}), $C_{\Delta}(S^n)$ is homotopy equivalent to the adjunction space $S^n \cup _{[\iota _n,\iota _n]} e^{2n}$, where $[\iota _n,\iota _n]: S^{2n-1} \rightarrow S^n$ represents the Whitehead bracket of the identity map $\iota _n: S^n \rightarrow S^n$ with itself. 
Taking this fact into consideration, we can easily obtain our next result:  

\begin{prop}\label{hc-sphere} The following statements hold:
	\begin{enumerate}
		\item[(a)]
		$\Sigma C_{\Delta}(S^n)\simeq S^{n+1}\vee S^{2n+1}$.
		
		\item[(b)] $\Sigma (C_{\Delta}(S^n)\wedge C_{\Delta}(S^n))\simeq S^{2n+1}\vee S^{3n+1}\vee S^{3n+1}\vee S^{4n+1}.$ In general, $\Sigma (C_{\Delta}(S^n)^{\wedge k})$ is a wedge of spheres where $S^{nk+1}$ is the one with minimun dimension, for all $k\geq 1$. 
	\end{enumerate}
\end{prop}
Observe that the suspension $\Sigma [\iota_n, \iota_n]$ is homotopical trivial so (a) is easily satisfied, whereas (b) is a direct consequence of (a) and the properties of the smash product and the suspension functors.

\bigskip
Let us delve into the analysis of the reduced diagonal map $$\bar{\Delta }:C_{\Delta }(S^n)\rightarrow C_{\Delta }(S^n)\wedge C_{\Delta }(S^n)$$ \noindent and its implications. By utilizing Lemma \ref{crudeHsquare} mentioned earlier, with $K=S^n$ and $\alpha =[\iota _n,\iota _n]$, we can decompose the reduced diagonal as follows:
$$\xymatrix{
	{C_{\Delta }(S^n)} \ar[d] \ar[rr]^(.4){\bar{\Delta }} & & {C_{\Delta }(S^n)\wedge C_{\Delta }(S^n)} \\
	{S^{2n}} \ar[rr]_{\bar{H}([\iota _n,\iota _n])} & & {S^{2n}} \ar[u] }$$
As observed in Remark \ref{Hil-inv}, $\bar{H}([\iota _n,\iota _n])$ is precisely the suspension of the map $$H_0([\iota _n,\iota _n]):S^{2n-1}\rightarrow S^{2n-1}.$$  Moreover, it is well-known that $H_0([\iota _n,\iota _n])=\pm h([\iota _n,\iota _n])\iota _{2n-1},$ where $h([\iota _n,\iota _n])\in \mathbb{Z}$ denotes the classical Hopf invariant of $[\iota _n,\iota _n]$, and that $h([\iota _n,\iota _n]])=0$ when $n$ is odd, and $h([\iota _n,\iota _n]])=2$ when $n$ is even.  Hence, $\bar{\Delta }:C_{\Delta} (S^n)\rightarrow C_{\Delta }(S^n)\wedge C_{\Delta }(S^n)$ is homotopically trivial when $n$ is odd. However, when $n$ is even, $\bar{H}([\iota _n,\iota _n]) :S^{2n}\rightarrow S^{2n}$ is a map whose degree is (up to sign) 2. Applying the suspension functor to the diagram above we can obtain another homotopy commutative diagram
$$\xymatrix{
	{S^{2n+1}\vee S^{n+1}} \ar[d] \ar[rr]^(.4){\Sigma \bar{\Delta }} & & {S^{2n+1}\vee S^{3n+1}\vee S^{3n+2}\vee S^{4n+1}} \ar@/^2.0pc/[d] \\
	{S^{2n+1}} \ar[rr]_{\pm 2Id } \ar@/^2.0pc/[u] & & {S^{2n+1}} \ar[u] }$$
\noindent where $\pm 2Id :S^{2n+1}\rightarrow S^{2n+1}$ denotes the map having degree $\pm $2, $S^{2n+1}\vee S^{3n+1}\vee S^{3n+2}\vee S^{4n+1}\rightarrow S^{2n+1}$ is a homotopy retraction of
$S^{2n+1}\rightarrow S^{2n+1}\vee S^{3n+1}\vee S^{3n+2}$ and $S^{2n+1}\rightarrow S^{2n+1}\vee S^{n+1}$ a homotopy section of $S^{2n+1}\vee S^{n+1}\rightarrow S^{2n+1}$.

\section{Product formulas}

In this section, we will explore the investigation of the Ganea conjecture concerning topological complexity. As previously mentioned, we will present sufficient conditions for a space $X$ to fulfill $\TC(X\times S^n)\geq \TC(X)+\TC(S^n)$ by employing the concepts of weak topological complexity and weak category in the sense of Berstein-Hilton, along with certain variations of such invariants. Considering the importance of the parity of the dimension in the sphere $S^n$, we have opted to analyze each case separately: odd-dimensional and even-dimensional. 

\subsection{Odd dimensional case}

To begin with, let us recall that for any space $X$ we can construct the homotopy cofiber of the diagonal map given by $$X\stackrel{\Delta }{\longrightarrow
}{X\times X}\stackrel{\rho _X}{\longrightarrow }C_{\Delta }(X)$$
\noindent where $\rho _X$ denotes the canonical map. It is worth noting that for any sphere $S^n$,
the projections $p_1:X\times S^n\rightarrow
X$ and $p_2:X\times S^n$ induce a natural map $\pi :C_{\Delta
}(X\times S^n)\rightarrow C_{\Delta }(X)\times C_{\Delta }(S^n).$
In fact, we can regard $\pi $ as the map induced by the homeomorphism
$id_X\times T\times id_{S^n}:X\times S^n\times X\times S^n \stackrel{\cong
}{\longrightarrow } X\times X\times S^n\times S^n$, where $T$ is
defined as $T(x,s):=(s,x)$. In particular, we have the following
commutative diagram
$$\xymatrix{
X\times S^n\times X\times S^n \ar[d]_{id_X\times T\times id_{S^n}}^{\cong}
\ar[rr]^{\rho_{X\times S^n}} &  & C_{\Delta}(X\times S^n) \ar[d]^{\pi }\\
X\times X\times S^n\times S^n \ar[rr]_{\rho_X \times \rho_{S^n}} &
& C_{\Delta}(X)\times C_{\Delta}(S^n) }$$

\vspace{-2.1cm}
\begin{flushright}
	$(*)$
\end{flushright}
\vspace{1cm}

The forthcoming lemma concerning the previously defined map $\pi$ is of great importance for our purposes. Its proof relies on the following well-known facts:
\begin{enumerate}
\item[(i)] If
$X\stackrel{f}{\longrightarrow }Y\stackrel{p}{\longrightarrow
}C_f$ is a homotopy cofibre sequence and $f$ admits a homotopy
retraction, then there exists $\sigma :\Sigma C_f\rightarrow
\Sigma Y$ a homotopy section of $\Sigma p.$
			
\item[(ii)] For any two connected spaces $X$, $Y,$ there is a natural homotopy equivalence
$\Sigma (X\times Y)\stackrel{\simeq}{\to} \Sigma X\vee \Sigma Y\vee \Sigma
(X\wedge Y).$ In particular, there exists a homotopy section for
the suspension of the projection map $X\times Y\rightarrow X\wedge Y.$
\end{enumerate}

\begin{lem}\label{hsection}
The suspension $\Sigma \pi: \Sigma C_{\Delta}(X\times S^n)\to
\Sigma (C_{\Delta}(X)\times C_{\Delta}(S^n))$ of $\pi $ admits a
homotopy section.
\end{lem}

\begin{proof}
	
For any space $Z,$ we take $\sigma_Z:\Sigma C_{\Delta}(Z) \to \Sigma (Z\times Z)$ a homotopy section of
$\Sigma\rho_Z.$ Then,
by the following decomposition of the map  $\Sigma (\rho_X\wedge \rho_{S^n})$:

	%

$$\xymatrix{
\Sigma((X\times X)\wedge (S^n\times S^n))\ar[d]^{\simeq} \ar@/^9.0pc/[ddddd]^{\Sigma (\rho_X\wedge \rho_{S^n})}\\
(\Sigma(X\times X))\wedge (S^n\times S^n)\ar[d]^{(\Sigma \rho_X)\wedge id} \\
(\Sigma C_{\Delta}(X))\wedge (S^n\times S^n)\ar[d]^{\simeq}\ar@/^2.0pc/[u]^{\sigma_X\wedge id}\\
C_{\Delta}(X)\wedge \Sigma (S^n\times S^n)\ar[d]^{id\wedge \Sigma \rho_{S^n}} \\
C_{\Delta}(X) \wedge \Sigma C_{\Delta}(S^n)\ar[d]^{\simeq}\ar@/^2.0pc/[u]^{id\wedge \sigma_{S^n}}\\
\Sigma (C_{\Delta}(X) \wedge  C_{\Delta}(S^n)) }$$ 
\noindent we see that $\Sigma (\rho_X\wedge \rho_{S^n})$ admits a homotopy section. As we have the homotopy equivalence $\Sigma (\rho_X \times \rho_{S^n})\simeq \Sigma\rho_X \vee \Sigma
\rho_{S^n}\vee \Sigma (\rho_X\wedge \rho_{S^n})$ we easily obtain
a homotopy section for $\Sigma (\rho_X \times \rho_{S^n}).$  We complete the proof by applying the suspension
functor to the diagram (*) above.
\end{proof}


\medskip
We are now in a position to state and prove our first main result in this section, specifically when the dimension of the sphere is odd:

\begin{thm}\label{result} Suppose $X$ is a $(p-1)$-connected CW-complex satisfying $\wTC(X)=k\geq 2$. If $\dim(X)\leq pk-1,$ then we have
$$\wTC(X\times S^n)\geq \wTC(X)+1$$ \noindent for all $n$.  If, in addition,
	$\wTC(X)=\TC(X)$, then $\TC(X\times S^n)\geq \TC(X)+1$ and hence, $\TC(X\times S^n)= \TC(X)+1$ when $n$ is odd.
\end{thm}


\begin{proof}
By utilizing the identity $\wTC(X)=\wcat(C_{\Delta}(X))$, we can establish that the $k$-th reduced diagonal $\bar{\Delta}_k: C_{\Delta}(X) \to C_{\Delta}(X)^{\wedge k}$ is not trivial. Furthermore, due to the inequality $\dim(X) \leq pk-1$, the map $\bar{\Delta}_k$ is stably non-trivial, that is, $\Sigma \bar{\Delta}_k \not\simeq *$ (equivalently, $\Sigma ^m\bar{\Delta}_k \not\simeq *$, for all $m$). Indeed, we have $\dim(C_{\Delta}(X)) = 2\dim(X)$ and $C_{\Delta}(X)^{\wedge k}$ is $(pk-1)$-connected, allowing us to apply the Freudenthal suspension theorem.



In order to establish the non-triviality of the $(k+1)$-th reduced diagonal $\bar{\Delta}_{k+1}:C_{\Delta}(X\times S^n)\to C_{\Delta}(X\times S^n)^{\wedge k+1}$, it is enough to demonstrate that its suspension is not homotopically trivial. To accomplish this, we first construct the following commutative diagram: 	
$$\xymatrix{
	C_{\Delta}(X\times S^n) \ar[r]^-{\bar{\Delta}_{k+1}} \ar[d]_{\pi }
	& C_{\Delta}(X\times S^n)^{\wedge k}\wedge C_{\Delta}(X\times S^n)\ar[d]\\
	C_{\Delta}(X)\times C_{\Delta}(S^n) \ar[d] \ar[r]^-{\bar{\Delta}_
			{k+1}} & (C_{\Delta}(X)\times C_{\Delta}(S^n))^{\wedge k}\wedge
		(C_{\Delta}(X)\times C_{\Delta}(S^n))\ar[d]\\
		C_{\Delta}(X)\wedge C_{\Delta}(S^n)\ar[r]^-{\bar{\Delta}_
			{k}\wedge id}& (C_{\Delta}(X))^{\wedge k}\wedge C_{\Delta}(S^n)
	}$$
Next, we apply the suspension functor to the diagram above, resulting in a new diagram where the curved arrows represent homotopy sections. Those are given by Lemma \ref{hsection} and the fact that, for any pair of connected spaces $X,Y$, the suspension of the projection map $X\times Y\rightarrow X\wedge Y$ always has a homotopy section:
	$$\xymatrix{
		\Sigma C_{\Delta}(X\times S^n) \ar[r]^-{\Sigma \bar{\Delta}_{k+1}} \ar[d]_{\Sigma \pi}
		& \Sigma C_{\Delta}(X\times S^n)^{\wedge k}\wedge C_{\Delta}(X\times S^n)\ar[d]\\
		\Sigma(C_{\Delta}(X)\times
		C_{\Delta}(S^n))\ar@/^2.0pc/[u]\ar[d]\ar[r]^-{\Sigma\bar{\Delta}_
			{k+1}} & \Sigma(C_{\Delta}(X)\times C_{\Delta}(S^n))^{\wedge
			k}\wedge
		(C_{\Delta}(X)\times C_{\Delta}(S^n))\ar[d]\\
		\Sigma C_{\Delta}(X)\wedge
		C_{\Delta}(S^n)\ar@/^2.0pc/[u]\ar[r]^-{\Sigma\bar{\Delta}_
			{k}\wedge id}& \Sigma (C_{\Delta}(X))^{\wedge k}\wedge
		C_{\Delta}(S^n) }$$ 
By using a simple diagram chase argument, where the homotopy sections must be considered, we can verify that the top map $\Sigma \bar{\Delta}_{k+1}$ is not homotopy trivial as long as the bottom map, $\Sigma \bar{\Delta}_{k}\wedge id$, is not homotopy trivial.
	
On the other hand, as we have that $\Sigma C_{\Delta}(S^{n})\simeq S^{n+1}\vee S^{2n+1}$ are homotopy equivalent  (see Proposition \ref{hc-sphere} (a)) such a bottom map, which is equivalent to 
$\bar{\Delta}_{k}\wedge \Sigma id:C_{\Delta}(X)\wedge \Sigma C_{\Delta}(S^n)\rightarrow  C_{\Delta}(X)^{\wedge k}\wedge \Sigma C_{\Delta}(S^n)$ 
can be, in turn, naturally identified with $\Sigma^{n+1}\bar{\Delta}_{k}\vee\Sigma^{2n+1}\bar{\Delta}_{k}$. Since $\bar{\Delta}_{k}:C_{\Delta}(X)\rightarrow C_{\Delta}(X)^{\wedge k}$ is stably non trivial, we conclude the result.
\end{proof}

\subsection{Even dimensional case}

To examine and analyze the scenario of even dimension, it is necessary to introduce a new notion for maps beforehand, which has certain similarities with both the conilpotency of a topological space and the weak sectional category of a map. 

 Remember that, for a given map $f:E\rightarrow B$ and any integer $n$, we use the notation $C_{\kappa_n}$ to represent the homotopy cofiber of the $n$-sectional fat wedge of $f$. Additionally, let $\Delta ^f_{n+1}:B\rightarrow C_{\kappa_n}$ denote the composite of the induced map $l_n$ into the cofiber with the diagonal map $\Delta _{n+1}^B:B\rightarrow B^{n+1}$:  
$$
\xymatrix{
{T^n(f)} \ar[r]^{\kappa_n} & {B^{n+1}} \ar[r]^{l_n} & {C_{\kappa_n}} \\
 & {B} \ar[u]^{ \Delta _{n+1}^B } \ar@{.>}[ur]_{\Delta ^f_{n+1}} & } 
$$

We will conveniently restrict ourselves to maps over a finite-dimensional CW-complex. Observe that, if $B$ is a finite dimensional CW-complex and $X$ is any space, then by the Freudenthal suspension theorem we have that
$$[\Sigma ^NB,\Sigma ^NX]\cong [\Sigma ^{N+1}B,\Sigma ^{N+1}X]\cong [\Sigma ^{N+2}B,\Sigma ^{N+2}X]\cong ...$$
\noindent for $N$ sufficiently large. In this case, $[\Sigma ^{\infty }B,\Sigma ^{\infty }X]\cong [\Sigma ^{N}B,\Sigma ^{N}X]$.

\begin{defin}
Let $f:E\rightarrow B$ be a map where $B$ is a finite dimensional CW-complex. The $2$-sigma weak sectional category of a space $f$, denoted as $\sigma \wsecat(f;2)$, is defined as the smallest non-negative integer $n$ for which the stable map
$$2\Sigma ^{\infty}(\Delta ^f_{n+1}):\Sigma ^{\infty}B\rightarrow \Sigma ^{\infty}C_{\kappa_n} $$
\noindent is trivial. If no such integer $k$ exists, then $\sigma \wsecat(f;2)=\infty .$ 
\end{defin}

It is evident from its definition that this homotopy invariant serves as a lower bound for the weak sectional category: $\sigma \mbox{wsecat}(f;2)\leq \mbox{wsecat}(f).$ Similarly to the approach taken with the weak sectional category \cite{Weaksecat}, this invariant can also be described in a more manageable equivalent form. Indeed, if we take the homotopy cofiber of $f$, ${E}\stackrel{f}{\longrightarrow }B\stackrel{j}{\longrightarrow }C_f$, then we can consider both the reduced diagonal $\bar{\Delta}^B_{n+1}:B\rightarrow B^{\wedge (n+1)}$ and the induced map $j^{\wedge (n+1)}:B^{\wedge (n+1)}\rightarrow C_f^{\wedge (n+1)}$. We can denote their composite as follows: $$\bar{\Delta }^f_{n+1}:=j^{\wedge (n+1)}\bar{\Delta}^B_{n+1}:B\rightarrow C_f^{\wedge (n+1)}$$

\begin{prop}
Let $f:E\rightarrow B$ be a map where $B$ is a finite dimensional CW-complex. Then, $\sigma \mbox{wsecat}(f;2)\leq n$ if, and only if, the stable map
$$2\Sigma ^{\infty }(\bar{\Delta }^f_{n+1}):\Sigma ^{\infty }B\rightarrow \Sigma ^{\infty }C_f^{\wedge (n+1)}$$ \noindent is trivial.
\end{prop}

\begin{proof}
By following the analogous steps outlined for the weak sectional category in \cite{Weaksecat}, one simply needs to consider the following homotopy commutative diagram. The top square of the diagram represents a homotopy pushout, and as a result, the induced map $C_{\kappa _n}\rightarrow C_f^{\wedge (n+1)}$ between the corresponding homotopy cofibers becomes a homotopy equivalence:
$$\xymatrix{
 & {T^n(f)} \ar[r] \ar[d]_{\kappa _n} & {T^n(C_f)} \ar[d] \\
 {B} \ar[r]^{\Delta _{n+1}^B} \ar[dr]_{\Delta ^f_{n+1}}  & {B^{n+1}} \ar[r]^{j^{n+1}} \ar[d]_{l_n} & {C_f^{n+1}} \ar[d]^{q^{C_f}_{n+1}} \\ 
 & {C_{\kappa _n}} \ar[r]^{\simeq } & {C_f^{\wedge (n+1)}}
 }$$ 
Applying conveniently the suspension functor we obtain the result. We leave the details to the reader.
\end{proof} 
A special case of $\sigma \mbox{wsecat}(-;2)$ is $\sigma \mbox{wcat}(-;2)$ for finite dimensional CW-complexes.

\begin{defin}
Let $B$ be a finite dimensional CW-complex. Then we define the $2$-sigma weak category of $B$, denoted as  $\sigma \wcat(B;2),$ as the sigma weak sectional category of the map $*\rightarrow B$. In other words, it is the smallest non-negative integer $n$ for which the stable map
$$2\Sigma ^{\infty} \bar{\Delta }^B_{n+1}:\Sigma ^{\infty}B\rightarrow \Sigma ^{\infty} B^{\wedge (n+1)}$$
	\noindent is trivial. If no such integer $n$ exists, then $\sigma \wcat(B;2)=\infty .$ 
  \end{defin}

Recall that the conilpotency of a space $B$, usually denoted as $\mbox{conil}(B)$, is defined as the least integer $n$ (or infinity) such that the suspension of the $(n+1)$-th reduced diagonal map, $\Sigma \bar{\Delta }_{n+1}^B:\Sigma B\rightarrow \Sigma B^{\wedge (n+1)}$, is homotopically trivial. By the very definition of $\sigma \mbox{wcat}(-;2),$ we have a chain of inequalities $$\sigma \mbox{wcat}(B;2)\leq \mbox{conil}(B)\leq \wcat (B)\leq \cat (B)$$

\begin{cor}
If $f:E\rightarrow B$ is a map where $B$ is a finite dimensional CW-complex, then
$$\sigma \wsecat(f;2)\leq \sigma \wcat(B;2)$$
\end{cor}

\medskip
We also have an analogous result to that given for weak category as in Proposition \ref{growth-of-wcat}. Indeed,
\begin{prop}\label{growth-of-sigmawcat}
Let $f:A\to B$ be a map between finite dimensional $(p-1)$-connected CW-complexes ($p\geq 1$). If $f$ is an $r$-equivalence, $\wcat(A)\geq k$ and
$\dim(A)\leq r+p(k-1)-1$, then $\sigma \wcat(B;2)\geq \sigma \wcat(A;2)$.
\end{prop}

\begin{proof}
We take $N$ sufficiently large so that $[\Sigma^{\infty}A,\Sigma^{\infty}A^{\wedge k}]\cong [\Sigma^{N}A,\Sigma^{N}A^{\wedge k}]$ and $[\Sigma^{\infty}B,\Sigma^{\infty}B^{\wedge k}]\cong [\Sigma^{N}B,\Sigma^{N}B^{\wedge k}]$. Since $f^{\wedge k}$ is an $(r+p(k-1))$-equivalence, we deduce that $\Sigma^Nf^{\wedge k}$ is an $(r+p(k-1)+N)$-equivalence. Furthermore, $\dim(\Sigma^NA)=\dim(A)+N$, and thus
$$(\Sigma ^Nf^{\wedge k})_*:[\Sigma ^{N}A,\Sigma ^{N}A^{\wedge k}]\rightarrow [\Sigma ^{N}A,\Sigma ^{N}B^{\wedge k}]$$ \noindent is injective. Now, since $\sigma wcat(A;2)\geq k$, we have $2\Sigma^N(\bar {\Delta}_k^A)\not \simeq *$. Taking into account the injection above and the following commutative diagram
$$\xymatrix{
{\Sigma ^NA} \ar[d]_{2\Sigma ^N(\bar{\Delta }_k^A)} \ar[rr]^{\Sigma ^Nf} & & {\Sigma ^NB} \ar[d]^{2\Sigma ^N(\bar{\Delta }_k^B)} \\
{\Sigma ^NA^{\wedge k}} \ar[rr]_{\Sigma ^N(f^{\wedge k})} & & {\Sigma ^NB} }$$ \noindent we conclude that $2\Sigma^N(\bar {\Delta}_k^B)\not \simeq *$, implying $\sigma \wcat(B;2)\geq k$. 
\end{proof}

Continuing the parallelism with the weak sectional category, we will now explore the relationship with the homotopy cofiber of the map when it admits a homotopy retraction. 

\begin{prop}
Let $f:E\rightarrow B$ be a map between finite dimensional CW-complexes. If $f$ admits a homotopy retraction, then
$$\sigma \wsecat(f;2)=\sigma \wcat(C_f;2)$$
\end{prop}

\begin{proof}

If $\sigma \mbox{wcat}(C_f;2)\leq n$, then $2\Sigma ^{\infty} \bar{\Delta }_{n+1}^{C_f}$ is trivial. Consequently, $2\Sigma ^{\infty} (\bar{\Delta }_{n+1}^{C_f}j)$ is also trivial. Considering the commutativity $\bar{\Delta }_{n+1}^{C_f}j=j^{\wedge (n+1)}\bar{\Delta}_{n+1}^B=\bar{\Delta }_{n+1}^f$ we can easily conclude that $\sigma \mbox{wsecat}(f;2)\leq n.$ 
 
Now suppose that $\sigma \mbox{wsecat}(f;2)\leq n$. We choose $N$ sufficiently large so that 
$$[\Sigma ^{\infty }B,\Sigma ^{\infty }C_f^{\wedge (n+1)}]\cong [\Sigma ^{N}B,\Sigma ^{N}C_f^{\wedge (n+1)}]$$ \noindent and 
$[\Sigma ^{\infty }C_f,\Sigma ^{\infty }C_f^{\wedge (n+1)}]\cong [\Sigma ^{N}C_f,\Sigma ^{N}C_f^{\wedge (n+1)}]$. Taking into account the commutative diagram
$$\xymatrix{
{\Sigma ^NB} \ar[d]_{\Sigma ^Nj} \ar[rr]^{2\Sigma ^N(\bar{\Delta }_{n+1}^B)} & & {\Sigma ^NB^{\wedge (n+1)}} \ar[d]^{\Sigma ^Nj^{\wedge (n+1)}} \\
{\Sigma ^NC_f} \ar[rr]_{2\Sigma ^N(\bar{\Delta }_{n+1}^{C_f})} & & {\Sigma ^N(C_f^{\wedge (n+1)})}
}$$ \noindent we have that the following composite is homotopically trivial
$$(2\Sigma ^N(\bar{\Delta }_{n+1}^{C_f}))(\Sigma ^Nj)=(\Sigma ^Nj^{\wedge (n+1)})(2\Sigma ^N(\bar{\Delta }_{n+1}^B))=2\Sigma ^N(\bar{\Delta }_{n+1}^f)\simeq *$$
Therefore, if we denote as $\delta :C_f\rightarrow \Sigma E$ the induced map in the homotopy cofiber of $j:B\rightarrow C_f,$ then there is a map $\phi$ making commutative the following diagram
$$\xymatrix{
{\Sigma ^NB} \ar[r]^{\Sigma ^Nj} & {\Sigma ^NC_f} \ar[r]^{\Sigma ^N \delta } \ar[dr]_{2\Sigma ^N(\bar{\Delta }_{n+1}^{C_f})} & {\Sigma ^{N+1}E} \ar@{.>}[d]^{\phi } \\
 & & {\Sigma ^NC_f^{\wedge (n+1)}}  }$$
However, from the Barrat-Puppe sequence
$$E\stackrel{f}{\longrightarrow}B\stackrel{j}{\longrightarrow }C_f\stackrel{\delta }{\longrightarrow }{\Sigma E}\stackrel{\Sigma f}{\longrightarrow }{\Sigma B}{\longrightarrow }\cdots$$ \noindent we deduce that $\delta \simeq (\Sigma r)(\Sigma f)\delta \simeq *,$ where $r$ stands for the homotopy retraction of $f.$ We conclude that $2\Sigma ^N(\bar{\Delta }^{n+1}_{C_f})\simeq *$, that is, $\sigma \mbox{wcat}(C_f;2)\leq n.$
\end{proof}

A particular case of $\sigma \wsecat(-;2)$ that is specially relevant in our work is as follows:

\begin{defin}
The $2$-sigma weak topological complexity of a finite dimensional CW-complex $X$ is defined as the $2$-sigma weak sectional category of the diagonal map $\Delta :X\rightarrow X\times X:$
$$\sigma \wTC(X;2):=\sigma \wsecat(\Delta ;2)$$
\end{defin}

As a direct consequence of this definition and all the previously established results, we obtain the following corollary:

\begin{cor}
Let $X$ be a finite dimensional CW-complex. Then 
$$\sigma \wTC(X;2)=\sigma \wcat(C_{\Delta }(X);2)\leq \wcat(C_{\Delta }(X))=\wTC(X)$$
\end{cor}

And now we are ready to state and prove the main result of this subsection, considering the even-dimensional case.

\begin{thm}\label{result2} Suppose $X$ is a finite dimensional $(p-1)$-connected CW-complex satisfying $\sigma \wTC(X;2)=k\geq 2$. If $\dim(X)\leq pk-1,$ then we have
$$\wTC(X\times S^n)\geq \sigma \wTC(X;2)+2$$ \noindent for all $n$ even.  If, in addition,
	$\sigma \wTC(X;2)=\TC(X)$, then $\TC(X\times S^n)=\TC(X)+\TC(S^n)$, for all $n$.
\end{thm}

\begin{proof}
Since $\sigma \mbox{wTC}(X;2)=\sigma \mbox{wcat}(C_{\Delta }(X);2)$, we can deduce that the map $$2\Sigma ^{\infty }(\bar{\Delta }_k):\Sigma ^{\infty }C_{\Delta }(X)\rightarrow \Sigma ^{\infty }C_{\Delta }(X)^{\wedge k}$$
\noindent is nontrivial. Additionally, applying the Freudenthal suspension theorem, we obtain that, $2\Sigma ^m(\bar{\Delta }_k):\Sigma ^{m}C_{\Delta }(X)\rightarrow \Sigma ^mC_{\Delta }(X)^{\wedge k}$ is not homotopy trivial for all $m$, considering that $\dim (X)\leq pk-1$. 

To establish the inequality $\wTC(X\times S^n)\geq 2$, we can follow the same steps as in Theorem \ref{result} for the case of odd dimensions. In this sense it suffices to demonstrate that the suspension of the map $\bar{\Delta }_{k+2}:C_{\Delta }(X\times S^n)\rightarrow C_{\Delta }(X\times S^n)^{\wedge (k+2)}$ is not homotopically trivial. To this aim we first construct a similar diagram
$$\xymatrix{
C_{\Delta}(X\times S^n) \ar[r]^-{\bar{\Delta}_{k+2}} \ar[d]_{\pi }
 & C_{\Delta}(X\times S^n)^{\wedge k}\wedge C_{\Delta}(X\times S^n)^{\wedge 2} \ar[d]\\
C_{\Delta}(X)\times C_{\Delta}(S^n) \ar[d] \ar[r]^-{\bar{\Delta}_
{k+2}} & (C_{\Delta}(X)\times C_{\Delta}(S^n))^{\wedge k}\wedge
(C_{\Delta}(X)\times C_{\Delta}(S^n))^{\wedge 2}\ar[d]\\
C_{\Delta}(X)\wedge C_{\Delta}(S^n)\ar[r]_-{\bar{\Delta}_
{k}\wedge \bar{\Delta}}& (C_{\Delta}(X))^{\wedge k}\wedge C_{\Delta}(S^n)^{\wedge 2}
}$$
By applying the suspension functor to this diagram, we obtain
$$\xymatrix{
\Sigma C_{\Delta}(X\times S^n) \ar[r]^-{\Sigma \bar{\Delta}_{k+2}} \ar[d]_{\Sigma \pi}
& \Sigma C_{\Delta}(X\times S^n)^{\wedge k}\wedge C_{\Delta}(X\times S^n)^{\wedge 2}\ar[d]\\
\Sigma(C_{\Delta}(X)\times C_{\Delta}(S^n))\ar@/^2.0pc/[u]\ar[d]\ar[r]^-{\Sigma\bar{\Delta}_{k+2}} & \Sigma(C_{\Delta}(X)\times C_{\Delta}(S^n))^{\wedge k}\wedge (C_{\Delta}(X)\times C_{\Delta}(S^n))^{\wedge 2}\ar[d]\\
\Sigma C_{\Delta}(X)\wedge C_{\Delta}(S^n)\ar@/^2.0pc/[u]\ar[r]^-{\Sigma\bar{\Delta}_
{k}\wedge \bar{\Delta}}& \Sigma (C_{\Delta}(X))^{\wedge k}\wedge
C_{\Delta}(S^n)^{\wedge 2} }$$
\noindent where the curved arrows are homotopy sections. For a completely analogous argument to the one made in the proof of Theorem \ref{result}, ultimately we must verify that the bottom map
$\Sigma \bar{\Delta }_k\wedge \bar{\Delta }:\Sigma C_{\Delta}(X)\wedge C_{\Delta}(S^n)\rightarrow \Sigma (C_{\Delta}(X))^{\wedge k}\wedge C_{\Delta}(S^n)^{\wedge 2} $, is non-trivial. This map can be considered, up to homeomorphism, as
$$\bar{\Delta }_k\wedge \Sigma \bar{\Delta }:C_{\Delta}(X)\wedge \Sigma C_{\Delta}(S^n)\rightarrow  (C_{\Delta}(X))^{\wedge k}\wedge \Sigma C_{\Delta}(S^n)^{\wedge 2}$$ Moreover, it produces, according to the argument at the end of Section 2, the following homotopy commutative diagram, where the curved arrow on the left represents a homotopy section, and the curved arrow on the right represents a homotopy retraction:
$$\xymatrix{
	{C_{\Delta}(X)\wedge (S^{2n+1}\vee S^{n+1}}) \ar[d] \ar[rr]^(.4){\bar{\Delta }_k\wedge \Sigma \bar{\Delta }} & & {C_{\Delta}(X)^{\wedge k}\wedge (S^{2n+1}\vee S^{3n+1}\vee S^{3n+2}\vee S^{4n+1})} \ar@/^2.0pc/[d] \\
	{C_{\Delta}(X)\wedge S^{2n+1}} \ar[rr]_{\bar{\Delta }_k\wedge \pm 2Id } \ar@/^2.0pc/[u] & & {C_{\Delta}(X)^{\wedge k} S^{2n+1}} \ar[u] }$$
Once again, the top arrow in the diagram is non-trivial in homotopy as long as the bottom arrow remains non-trivial in homotopy. However, the bottom arrow can be equivalently expressed as $\pm 2\Sigma ^{2n+1}(\bar{\Delta }_k):\Sigma ^{2n+1}C_{\Delta }(X)\rightarrow \Sigma ^{2n+1}C_{\Delta}(X)^{\wedge k},$ which, according to our assumptions, is not homotopy trivial.
The remainder of the proof is straightforward. Note that the equality
$\sigma \wTC(X;2)=\TC(X)$ implies $\wTC(X)=\TC(X)$. Therefore the equality $\TC(X\times S^n)=\TC(X)+\TC(S^n)$ for $n$ odd follows from Theorem \ref{result}. 
\end{proof}

\section{Examples: two-cell complexes}

We will now apply our results in order to exhibit a new class of spaces for which the equality
$$\mbox{TC}(X\times S^{n})= \mbox{TC}(X)+ \mbox{TC}(S^n)$$
holds for any $n$.

We first note that, since the zero-divisor cuplength over a field ${\Bbbk}$ of $X\times S^ n$ is at least ${\rm zcl}_{\Bbbk}(X)+1$, any space $X$ satisfying ${\rm zcl}_{\Bbbk}(X)=\TC(X)$ satisfies $\TC(X\times S^ n)\geq \TC(X)+1$ for any $n$ and  $\TC(X\times S^ n)=\TC(X)+1$ when $n$ is odd. As is well-known, there are a lot of spaces for which the equality ${\rm zcl}_{\Bbbk}(X)=\TC(X)$ holds, for instance the orientable surfaces, the $1$-connected symplectic manifolds... When $n$ is even and $2$ is invertible in the field ${\Bbbk}$, we have ${\rm zcl}_{\Bbbk}(X\times S^n)={\rm zcl}_{\Bbbk}(X)+2$. Consequently, the same reasoning shows that, any space $X$ satisfying ${\rm zcl}_{\Bbbk}(X)=\TC(X)$ where $2$ is invertible in the field ${\Bbbk}$, satisfies $\TC(X\times S^ n)= \TC(X)+\TC(S^n)$ for any $n$. 

We are therefore interested in spaces for which ${\rm zcl}_{\Bbbk}(X)<\wTC(X)=\TC(X)$. As mentioned before, the space $X=S^3\cup _{\alpha }e^7$, where $\alpha\in \pi_6(S^3)$ is the Blakers-Massey element, satisfies
${\rm zcl}_{\Bbbk}(X)=2$ and $\wTC(X)=\TC(X)=3.$
As $X$ is $2$-connected and $\mbox{dim}(X)=7\leq 8$, Theorem \ref{result} guarantees that
$$\mbox{TC}(X\times S^{n})\geq \mbox{TC}(X)+1$$ \noindent and that the
equality holds for $n$ odd. This space is a two-cell complex $S^p\cup e^{q+1}$ in the metastable range, which means that $2p-1 < q < 3p-2$.


By using results from \cite{GonzalezGrantV} and \cite{GGV2}, we will establish a more general result on two-cell complexes $X=S^p\cup_{\alpha}e^ {q+1}$ in the metastable range. The analysis of $\cat(C_{\Delta}(X))$ for such spaces by means of Hopf invariants\footnote{Actually the $\cat$-Hopf invariants used in \cite{GonzalezGrantV} and \cite{GGV2} are based on Iwase's approach and defined using the characterization of the LS-category in terms of Ganea fibrations. The information contained in these invariants is nevertheless equivalent to the information contained in the Berstein-Hilton invariants considered here.} realized in \cite{GGV2} will be especially relevant for us. Let us recall from Remark \ref{Hil-inv} that, for $\alpha:S^q\to S^p$ in the metastable range, $H(\alpha)$ is a stable homotopy class in $\pi_q(\Omega S^p\ast \Omega S^p)$ which is completely determined by its projection $H_0(\alpha)$ on the bottom sphere $S^{2p-1}$ of $\Omega S^p\ast \Omega S^p\simeq S^{2p-1}\vee  S^{3p-2} \vee  S^{3p-2} \vee \cdots $.

From \cite{GGV2} we have a cone decomposition of $C_{\Delta}(X)$
\begin{equation}\label{conedecomposition}
\ast=D_0\subset D_1\subset D_2\subset D_3\subset D_4=C_{\Delta}(X)
\end{equation}
with attaching maps of the form $S_i \stackrel{\beta_i}{\longrightarrow}D_{i}$
where
$$S_0=S^{p-1}, \qquad S_1=S^{2p-1}\vee S^q, \qquad S_2=S^{p+q}\vee S^{p+q} \quad \mbox{and}\quad S_3=S^{2q+1}.$$
In particular $D_1=S^p$ and $D_2=S^p\cup e^{2p}\cup e^{q+1}$. For each $i$, $\cat(D_i)\leq i$ and it is shown in \cite{GGV2} that the relevant Hopf invariant for the study of $\cat(D_i)$, here denoted by $H(\beta_i)$, can be expressed in terms of $H_0(\alpha)$.

We will here use this information to estimate $\wTC(X)=\wcat(C_{\Delta}(X))$ and $\sigma\wTC(X;2)=\sigma\wcat(C_{\Delta}(X;2))$.
We first establish:

\begin{thm}\label{twocell}
Let $X=S^p\cup_{\alpha}e^ {q+1}$ be a two-cell complex such that $2p-1<q<3p-2$ and $(2+(-1)^ p)H_0(\alpha)\neq 0$. Then
\begin{enumerate}
	\item[(a)] $\wTC(X)\geq 3$
	\item[(b)] $\wTC(X)=\TC(X)$
\end{enumerate}
Consequently, for any odd $n$, $\TC(X\times S^n)= \TC(X)+1$.
\end{thm}

\begin{proof}

Considering the cone decomposition (\ref{conedecomposition}), we first note that $\wcat(D_2)=\cat(D_2)=2$. Indeed the map $X\stackrel{i_1}{\to}X\times X \to C_{\Delta}(X)$ induces a map $X\to D_2$ which is a $(2p-1)$-equivalence and the result follows from $\wcat(X)=2$ and Proposition \ref{growth-of-wcat}.

As $p+q\geq \dim D_2=q+1$ and $\dim D_2\leq 3p-2$ we apply Lemma \ref{crudeHsquare} to the cofibration sequence
\[S_2 \stackrel{\beta_2}{\longrightarrow}D_{2}\hookrightarrow D_{3}.\]
 We then obtain the following homotopy commutative diagram:
	\[\xymatrix{D_3 \ar[rr]^{\bar{\Delta}}\ar[d]_{\delta} && D_3^{\wedge 3}\\
	\Sigma S_2 \ar[rr]_{\bar H(\beta_2)}&& D_2^{\wedge 3}\ar@{^(->}[u]}\]
and $\wcat(D_3)\leq 2$ if and only if $\bar H(\beta_2)=0$. Note that $$\Sigma S_2=S^{p+q+1}\vee S^{p+q+1} \quad \mbox{and} \quad  D_2^{\wedge 3}\simeq S^{3p}\cup \bigcup\limits_{3}e^{4p}\cup \dots.$$
Note also that the $\cat$-Hopf invariant $H(\beta_2)$ is given by a map
$$S_2=S^{p+q}\vee S^{p+q} \to F_2(D_2)=\Omega D_2*\Omega D_2*\Omega D_2\simeq S^{3p-1}\cup \bigcup\limits_{3}e^{4p-2}\cup \dots.$$
Since $p+q<4p-2$, this map can only touch the bottom sphere $S^{3p-1}$ of $F_2(D_2)$.
It follows from \cite[Corollary 4.8]{GGV2} that $H(\beta_2)$ is completely determined by the stable class $(2+(-1)^ p)H_0(\alpha)$. Actually it follows from the analysis done in the proof of \cite[Corollary 4.8]{GGV2} and \cite[Theorem 5.6]{GonzalezGrantV} that, on each sphere $S^{p+q}$, $H(\beta_2)$ can be identified to $$\Sigma^{p}((2+(-1)^ p)H_0(\alpha)): S^{p+q}\to S^{3p-1}.$$
Taking into account Remark \ref{rem-inv}, we can then conclude that on each sphere $S^{p+q+1}$, the crude invariant $\bar H(\beta_2)$ can be identified to $$\Sigma^{p+1}((2+(-1)^ p)H_0(\alpha)): S^{p+q+1}\to S^{3p}.$$
Since $(2+(-1)^ p)H_0(\alpha)$ is stable and nontrivial, we conclude that $\wcat(D_3)\geq 3$ and, by Proposition \ref{growth-of-wcat}, that $\wTC(X)=\wcat(C_{\Delta})\geq 3$. 

We now prove (b). If $\TC(X)=3$ this is immediate. Let us suppose that $\TC(X)=4$. Since $D_3$ is a $3$-cone, we have $\cat(D_3)\leq  3$ (actually $\cat(D_3)= 3$ since $\wcat(D_3)\geq  3$). As we have $\dim (D_3)=p+q+1 \leq 4p-2$ and $\dim (D_3) \leq 2q+1$ we can apply Lemma \ref{crudeHsquare} to the cofibration sequence
\[S_3 \stackrel{\beta_3}{\longrightarrow}D_{3}\hookrightarrow D_{4}=C_{\Delta}(X).\]
We then obtain the following homotopy commutative diagram:
\[\xymatrix{D_4 \ar[rr]^{\bar{\Delta}}\ar[d]_{\delta} && D_4^{\wedge 4}\\
	\Sigma S_3 \ar[rr]_{\bar H(\beta_3)}&& D_3^{\wedge 4}\ar@{^(->}[u]}\]
and $\wcat(D_4)\leq 3$ if and only if $\bar H(\beta_3)=0$. As before we observe that $$\Sigma S_3=S^{2q+2} \quad \mbox{and} \quad  D_3^{\wedge 4}\simeq S^{4p}\cup \bigcup\limits_{4}e^{5p}\cup \dots$$ and that the $\cat$-Hopf invariant $H(\beta_3)$ is given by a map
$$S_3=S^{2q+1} \to F_3(D_3)=\Omega D_3*\Omega D_3*\Omega D_3*\Omega D_3\simeq S^{4p-1}\cup \bigcup\limits_{4}e^{5p-2}\cup \dots.$$
In this case, dimensional reasons do not permit us to see that $H(\beta_3)$ can only touch the bottom sphere $S^{4p-1}$ of $F_3(D_3)$. However, in \cite[Proof of Theorem 1.1]{GGV2}, it has been proved that $H(\beta_3)$ coincides with the map
\begin{equation*}\xymatrix{
	S^{2q+1} \ar[rrrr]^{2(2+(-1)^p)\cdot H_0(\alpha)\ast H_0(\alpha)} &&&& S^{4p-1} \ar@{^{(}->}[r] & F_3(D_3)
}\end{equation*}
where the second map is the inclusion of the bottom sphere. 
Furthermore, it follows from \cite[Lemma 4.9]{GGV2} that the bottom sphere $S^{4p-1}$ of the $(3p+q-1)$-skeleton of $F_3(D_3)$ splits off as a wedge summand, which permits one to say that $H(\beta_3)$ is completely determined by the (stable) map $h:=2(2+(-1)^ p)H_0(\alpha)\ast H_0(\alpha)$. We will see that the same is true for $\bar H(\beta_3)$. Indeed, since the inclusion $S^p\hookrightarrow X=S^p\cup_{\alpha} e^{q+1}$ is a $q$-equivalence, it induces a $q$-equivalence $C_{\Delta}(S^p)\to C_{\Delta}(X)$ and we can see that all the vertical maps in the following diagram are $(3p+q)$-equivalence.
\[
\xymatrix{
&& \Sigma F_3(C_{\Delta}(S^p)) \ar[r]_{ev\circ\Sigma \bar q_{4}}\ar[d] & C_{\Delta}(S^p)^{\wedge 4}\ar[d]\\
S^{2q+2} \ar[r]^{\Sigma h}  \ar@/_1.0pc/[rr]_{\Sigma H(\beta_3)}& S^{4p} \ar@{^{(}->}[r]  \ar@{^{(}->}[ru]& \Sigma F_3(D_3) \ar[r]_{ev\circ\Sigma \bar q_{4}}  & D_3^{\wedge 4}
}
\]
Recall that the composition of the bottom line coincides with $\bar H(\beta_3)$ and that this element is stable.
As $\Sigma C_{\Delta}(S^p)^{\wedge 4}\to \Sigma D_3^{\wedge 4}$ is a $(3p+q+1)$-equivalence, we can deduce from Proposition \ref{hc-sphere} that the bottom sphere $S^{4p+1}$ of the $(3p+q+1)$-skeleton of $\Sigma D_3^{\wedge 4}$ splits off as a wedge summand. Furthermore, as indicated in the diagram, the inclusion of the bottom sphere of $\Sigma F_3(C_{\Delta}(X))$ lifts into $\Sigma F_3(C_{\Delta}(S^p))$ and, by using Remark \ref{rem-inv}, we can see that all the maps from $S^{4p}$ correspond to the inclusion of the bottom sphere. Since $2q+2<3p+q$, we can then see, after suspension of the diagram, that $\Sigma \bar H(\beta_3)$ is completely determined by $\Sigma^2 h$. By stability, 
  we can then conclude that $\bar H(\beta_3)$ is completely determined by $2(2+(-1)^ p)H_0(\alpha)\ast H_0(\alpha)$. Since $\TC(X)=4$, this element is nontrivial (\cite[Corollary 4.10]{GGV2}) and consequently $\wcat(D_4)\geq 4$. Therefore $\wcat(C_{\Delta}(X))=\wTC(X)=4$ and $  \wTC(X)=\TC(X)$. The final statement follows from Theorem \ref{result}.
\end{proof}

We now exhibit a class of spaces which satisfy $\TC(X\times S^n)= \TC(X)+\TC(S^n)$ for any $n$.

\begin{thm}\label{twocell2} Let $X=S^p\cup_{\alpha}e^ {q+1}$ be a two-cell complex such that $2p-1<q<3p-2$ ($p\geq 2$) and $2(2+(-1)^ p)H_0(\alpha)\neq 0$. Then
	\begin{enumerate}
		\item[(a)] $\sigma\wTC(X;2)\geq 3$
		\item[(b)] Moreover, if one of the following conditions holds
		\begin{enumerate}
			\item [(i)] $2(2+(-1)^p)H_0(\alpha)\ast H_0(\alpha)= 0$
				\item [(ii)] $4(2+(-1)^p)H_0(\alpha)\ast H_0(\alpha)\neq 0$
		\end{enumerate} then $\sigma\wTC(X;2)=\TC(X)$ and, for any $n$, $$\TC(X\times S^n)= \TC(X)+\TC(S^n).$$		
	\end{enumerate}
\end{thm}

We will follow the same strategy as before by using the following adaptation of Lemma \ref{crudeHsquare}.

\begin{lem}\label{crudeHsquare2} Let $K$ be a $(p-1)$-connected CW complex with $p\geq 2$, $\cat(K)\leq k$ and $\dim (K) \leq (k+1)p-2$. Let $d\geq \dim (K)$ and $\alpha: S\to K$ be a map where $S$ is a finite wedge of spheres $S^d$. Then, for $X=K\cup_{\alpha }CS$ and any integer $N\geq 0$, there exists a homotopy commutative diagram
\[\xymatrix{\Sigma^N X \ar[rr]^{2\Sigma^N\bar{\Delta}_{k+1}}\ar[d]_{\Sigma^N \delta} && \Sigma^N X^{\wedge (k+1)}\\
\Sigma^{N+1} S \ar[rr]_{2\Sigma^N\bar H(\alpha)}&&\Sigma^N K^{\wedge (k+1)}\ar@{^(->}[u]}\]
where $\delta$ is the connecting map in the Puppe sequence of the cofibration $S\stackrel{\alpha}{\to}K\to X$. Moreover $\sigma\wcat(X;2)\leq k$ if and only if the stable class of $2\bar H(\alpha)$ is trivial.
\end{lem}

\begin{proof}
The diagram is obtained from the diagram of Lemma \ref{crudeHsquare} and the proof of the last statement follows from a direct adaptation of the arguments given in the proof of Lemma \ref{crudeHsquare}.
\end{proof} 

\begin{rem} {\rm Applying Lemma \ref{crudeHsquare2} to $X=S^p\cup_{\alpha}e^ {q+1}$ where $2p-1<q<3p-2$ and $2H_0(\alpha)\neq 0$, we obtain that $\sigma\wcat(X;2)\geq 2$ since $2H_0(\alpha)$ is a stable class. In this case we will have $\sigma\wcat(X)=\wcat(X)=2$. However if $H_0(\alpha)\neq 0$ but $2H_0(\alpha)= 0$ then $X$ satisfies $\sigma\wcat(X;2)=1<\wcat(X)=2$. }
\end{rem}

\begin{proof}[Proof of Theorem \ref{twocell2}]
The proof is parallel to the proof of Theorem \ref{twocell}. We first note that, since the hypothesis implies that $\sigma\wcat(X;2)=\wcat(X)=2$, it follows from Proposition \ref{growth-of-sigmawcat} that the stage $D_2$ of the cone decomposition (\ref{conedecomposition}) satisfies $\sigma\wcat(D_2;2)=\wcat(D_2)=2$.
Applying Lemma \ref{crudeHsquare2} to the cofibration sequence
\[S_2 \stackrel{\beta_2}{\longrightarrow}D_{2}\hookrightarrow D_{3}\]
we see that $\sigma\wcat(D_3;2)\leq 2$ if and only if $2\bar H(\beta_2)$ is stably trivial. Since this class is completely determined by the stable class $2(2+(-1)^ p)H_0(\alpha)$, the hypothesis $2(2+(-1)^ p)H_0(\alpha)\neq 0$ permits us to conclude that $\sigma\wcat(D_3;2)\geq 3$. It then follows from Proposition \ref{growth-of-sigmawcat} that $\sigma\wTC(X;2)=\sigma\wcat(C_{\Delta}(X);2)\geq 3$ as claimed in $(a)$. By \cite{GonzalezGrantV}, we know that the condition given in $(i)$ of $(b)$ implies that $\TC(X)\leq 3$ so we have $\sigma\wTC(X;2)=\TC(X)$ in that case. In case $(ii)$, the equality $\sigma\wTC(X;2)=\TC(X)$ follows from the application of Lemma \ref{crudeHsquare2} to the cofibration sequence 
\[S_3 \stackrel{\beta_3}{\longrightarrow}D_{3}\hookrightarrow D_{4}=C_{\Delta}(X)\]
together with the fact that the stable class of $2\bar H(\beta_3)$ is completely determined by the stable class $4(2+(-1)^p)H_0(\alpha)\ast H_0(\alpha)$.
\end{proof}

\begin{ex} {\rm As an application of Theorem \ref{twocell2}, we consider the space $X=S^8\cup_{\alpha} e^{19}$ where $\alpha:S^{18}\to S^8$ is the composition of a generator $\gamma$ of the stable group $\pi_{18}(S^{15})=\Z/24\Z$ with the Hopf map $S^{15}\to S^8$. The map $\alpha$ is in the metastable range and $H_0(\alpha)=\gamma$ (see, for instance \cite[Corollary 6.23]{C-L-O-T}). Consequently $3H_0(\alpha)$ and $6H_0(\alpha)$ are not trivial while $6H_0(\alpha)\ast H_0(\alpha)=0$ since $H_0(\alpha)\ast H_0(\alpha)\in \pi_{37}(S^{31})=\Z/2\Z$. We then obtain $\TC(X\times S^n)=\TC(X)+\TC(S^n)$ for any $n$.}
\end{ex}

\begin{rem} {\rm 
As mentioned in the introduction of \cite{GGV2}, the first possible instance $(q,p)$ of a map $\alpha: S^q\to S^p$ in the metastable range satisfying $H_0(\alpha)\neq 0$ but $3H_0(\alpha)=0$ is $(q,p)=(14,6)$. We here construct such a map. We consider the following part of the EHP sequence 
together with the isomorphism given by the double suspension $E^2$ and the description of the composite $PE^2$ (see \cite{Whitehead}, p. 548)
\[
\xymatrix{\pi_{13}(S^5) \ar[r]^E & \pi_{14}(S^6) \ar[r]^H &\pi_{14}(S^{11}) \ar[r]^P & \pi_{12}(S^{5})\\
&&\pi_{12}(S^9)\ar[u]^{E^2}_{\cong} \ar[ur]_{[\iota_5,\iota_5]\circ (-)}
}
\]
We note that the map $H$ corresponds to our $H_0$.
Let $\gamma \in \pi_{12}(S^9)=\Z/24\Z$ be a generator and $\beta=E^2(8\gamma)$. We have 
\[P\beta= [\iota_5,\iota_5]\circ 8 \gamma = 2[\iota_5,\iota_5]\circ 4 \gamma=0\]
since $[\iota_5,\iota_5]$ is of order $2$. Note that the second equality holds because any map in $\pi_{12}(S^9)$ is a suspension and therefore a co-H-map. Consequently, there exists $\alpha \in \pi_{14}(S^6) $ such that $H(\alpha)=\beta\neq 0$. However, $3H(\alpha)=3\beta=E^2(24 \gamma)=0$. In conclusion, $X=S^6\cup_{\alpha}e^{15}$ is a two-cell complex in the metastable range which does not satisfy the conditions of Theorem \ref{twocell}. Therefore it might be a candidate for an example of a two-cell space satisfying $\TC(X\times S^n)=\TC(X)$ for $n$ odd.

}\end{rem} 

\section*{Acknowledgments} 
	
	This research was partially financed by Portuguese Funds through FCT (Funda\c c\~ao para a Ci\^encia e a Tecnologia) within the Projects UIDB/00013/2020 and UIDP/00013/2020, and by the project PID2020-118753GB-I00 of the Spanish Ministry of Science and Innovation.


\bigskip

\begin{thebibliography}{99}

\bibitem{B-H} {I. Berstein and P.J. Hilton}. Category and
generalized Hopf invariants. \textit{Illinois J. Math.} 4 (1960),
437-451

\bibitem{carrasquel}
{J. G. Carrasquel-Vera.}
The Ganea conjecture for rational approximations of sectional category,
\textit{J. Pure Appl. Algebra},  \textbf{220} (2016), no. 4, 1310--1315.

\bibitem{C-L-O-T} O. Cornea, G. Lupton, J. Oprea and D. Tanr\'e.
\textit{Lusternik-Schnirelmann category}. Math. Surveys and
Monographs, vol. \textbf{103}, AMS, 2003.

\bibitem{Dr}{ A. Dranishnikov.} The topological complexity and the homotopy cofiber of the diagonal map for non-orientable surfaces.
\textit{Proc. Amer. Math. Soc.} \textbf{144} (2016), 4999--5014.

\bibitem{Far} {M. Farber.} Topological complexity of motion
planning. \textit{Discrete Comput. Geom.} \textbf{29} (2003),
211-221.

\bibitem{Weaksecat}{J. M. Garc\'{\i}a Calcines and L.
Vandembroucq}. Weak sectional category. \textit{Journal of the
London Math. Soc.} \textbf{82}(3) (2010), 621-642.

\bibitem{GC-V}{J. M. Garc\'{\i}a Calcines and L.
Vandembroucq}. Topological complexity and the homotopy cofibre of
the diagonal map. \textit{Math. Z.} \textbf{274}(1-2) (2013),
145-165.

\bibitem{Ganea1} {T. Ganea.}  Some problems on numerical homotopy invariants.
\textit{In: Symposium on Algebraic Topology , 1971, Lecture Notes in Math.} \textbf{249} (1971), 23--30.

\bibitem{GonzalezGrantV} {J. Gonz\'alez, M. Grant and L. Vandembroucq.} Hopf invariants for sectional category with applications to topological robotics. \textit{Q. J. Math.} 70 (2019), no. 4, 1209--1252.

\bibitem{GGV2}  {J. Gonz\'alez, M. Grant, and L. Vandembroucq} Hopf invariants, topological complexity, and LS-category of the cofiber of the diagonal map for two-cell complexes. \emph{Topological complexity and related topics}, 133--150, Contemp. Math., 702, Amer. Math. Soc., Providence, RI, 2018.


\bibitem{Hess} K. Hess. A proof of Ganea's conjecture for rational spaces,
\textit{Topology} \textbf{30} (1991), no. 2, 205--214.

\bibitem{Iwase} N.  Iwase.  Ganea's conjecture on Lusternik-Schnirelmann
category, \textit{Bull.  London Math.  Soc.} \textbf{30} (1998) 1--12.

\bibitem{Ainfinity} N.  Iwase.  $A_{\infty}$ methods in Lusternik-Schnirelmann category, \textit{Topology} \textbf{41} (2002), 695--723.

\bibitem{Jessup} B. Jessup. Rational approximations to L-S category and a conjecture of Ganea,
\textit{Trans. Amer. Math. Soc.} \textbf{317} (1990), no. 2, 655--660.

\bibitem{JMP} B. Jessup, A. Murillo and P.-E. Parent. Rational Topological Complexity.
\textit{Algebraic \& Geometric Topology} \textbf{12} (2012),
1789-1801.

\bibitem{Rudyak} {Y. B. Rudyak.} On the Ganea conjecture for manifolds,
\textit{Proc. Amer. Math. Soc.} \textbf{125} (1997), no. 8, 2511--2512.

\bibitem{Singhof}. W. Singhof, Minimal coverings of manifolds with balls,
\textit{Manuscripta Math.} \textbf{29} (1979), no. 2-4, 385--415.

\bibitem{strom-ls-on-manifolds} {J. Strom}. Two special cases of Ganea's conjecture,
\textit{Trans.~Amer.~Math.~Soc.} \textbf{352} (2000), no.~2, 679--688.

\bibitem{Sch} {A. Schwarz.} \textit{The genus of a fiber space},
A.M.S. Transl. 55 (1966), 49-140.

\bibitem{Stanley} {D. Stanley.} Spaces and Lusternik-Schnirelmann category n and cone length n+1, \textit{Topology} \textbf{39} (2000), no.5, 985--1019.

\bibitem{Stanley2} {D. Stanley.} On the Lusternik--Schnirelmann category of maps. \textit{Canad. J. Math.} \textbf{54} (2002), 608--633.

\bibitem{Qcat} L.  Vandembroucq,  Fibrewise suspension and Lusternik-Schnirelmann category, \textit{To\-po\-lo\-gy} \textbf{41} (2002), 1239--1258.

\bibitem{Whitehead} G. W. Whitehead. \textit{Elements of homotopy theory}, Graduate Texts in Mathematics, vol. 61, Springer-Verlag, New York, 1978.



\end{thebibliography}
\end{document}